\documentclass[12pt]{amsart}

\usepackage[initials]{amsrefs}
\usepackage{amssymb,amsmath,amsfonts,latexsym,amsthm,paralist,graphicx}
\usepackage{xcolor}

\usepackage{amscd}

\usepackage{enumerate}

\usepackage{stackrel}

\newtheorem{theorem}{Theorem}[section]

\newtheorem{proposition}[theorem]{Proposition}
\newtheorem{corollary}[theorem]{Corollary}

\newtheorem{definition}{Definition}[section]

\newtheorem{openproblem}{Open Problem}[]

{\theoremstyle{definition}
\newtheorem{remark}[theorem]{Remark}
\newtheorem{example}[theorem]{Example}

}

\def\N{\mathbb N}

\def\R{\mathbb R}

\title[{\tiny Infinite pointwise lineability: general criteria and applications}]{Infinite pointwise lineability: general criteria and applications}

\author{M.C.~Calder\'on-Moreno, P.J.~Gerlach-Mena, J.A.~Prado-Bassas}

\thanks{The first and third authors
have been partially supported by the Plan Andaluz de Investigaci\'on de la Junta de Andaluc\'{\i}a FQM-127 and by FEDER grant US-1380969.}
\date{}

\begin{document}

\maketitle

\begin{abstract}
\noindent In this paper we introduce the concept of infinite pointwise dense lineability (spaceability), and provide a criterion to obtain density from mere lineability. As an application, we study the linear and topological structures within the set of infinite differentiable and integrable functions, for any order $p \geq 1$, on $\R^N$ which are unbounded in a pre-fixed set.
\end{abstract}

\section{Introduction}

Lineability, introduced by V.I. Gurariy \cite{gurariy1966}, studies the existence of linear structures within sets with nonlinear properties. Formally, a subset $M$ of a vector space $X$ is $\alpha$-lineable if $M \cup \{0\}$ contains an $\alpha$-dimensional subspace $W$ of $X$, where $\alpha$ denotes any cardinal number. If additionally $X$ is endowed with a topology and $W$ is dense in $X$ (respectively, closed) we say that $M$ is $\alpha$-dense lineable (respectively, $\alpha$-spaceable) in $X$.

In the last years many examples about the existence of such structures have been provided (see \cite{aronbernalpellegrinoseoane2016}, \cite{bernalpellegrinoseoane2014}).

In fact, the search was initially for the existence of linear structures in specific cases of known spaces such as, for example, continuous nowhere differentiable functions \cite{weierstrassmonsters}, everywhere surjective functions \cite{everywheresurjective} or $p$-integrable functions that are not $q$-integrable for any $q\le p$ \cite{bernal2010}.

Recently, there has been a shift towards searching for more general results related to new linear behaviors, such as vector spaces containing any pre-fixed vector or algebraic structures of higher dimension than a given one. Furthermore, these results often come with applications to specific cases. Some of the results obtained in this regard can be found in \cite{dinizraposo, favaropellegrinotomaz,leonetirusosomaglia,papathanasiou}.

To be more concrete, D. Pellegrino and A. Raposo Jr. \cite{pellegrinoraposo2021} introduced a pointwise type of lineability as follows:

A subset $M$ of a (topological) vector space $X$ is called pointwise $\alpha$-(dense) lineable if for each $x \in M$, there is a (dense) $\alpha$-dimensional subspace $W_x$ such that
$$x \in W_x \subset M \cup \{0\}.$$
If $W_x$ is a closed $\alpha$-dimensional subspace, we say that $M$ is pointwise $\alpha$-spaceable. If $\alpha = {\rm dim}(X)$, we say that $M$ is maximal pointwise (dense) lineable (spaceable). It is clear that these pointwise notions imply the respective first ones, and that both concepts of (pointwise) dense lineability and spaceability are (strictly) stronger than mere lineability. Moreover there are only a few results which provide some (sufficient) additional conditions to ``jump" to density or spaceability from lineability.

In this paper we introduce the concept of {\em infinite pointwise (dense) lineability (spaceability)}, which relates to the existence of infinitely many vector spaces of infinite dimension with the above definitions. Within these, we provide criteria that allow us to obtain denseness of the corresponding vector spaces from mere lineability, which will be a helpful tool to obtain existence of large linear structures within certain families of functions.

As an application, we consider the family of infinitely differentiable, integrable functions on $\R^N$ which are unbounded on a pre-fixed set, and we show its maximal infinite pointwise (dense) lineability, as well as its spaceability. With this we continue and complete a number of previous and recent results about the set of continuous, unbounded and integrable functions on $[0, + \infty)$ (see \cite{calderongerlachprado2019}, \cite{favaropellegrinoraposoribeiroPREPRINT}).

\section{Infinite pointwise lineability: main definitions and general criteria}

Inspired by \cite{konidasPREPRINT}, \cite{konidasnestoridisPREPRINT}, and the notion of pointwise $\alpha$-lineability we introduce the following definition.


\begin{definition} \label{infinitepointwisedefinition}
Let $X$ be a vector space, $\alpha$ an infinite cardinal number and $M \subset X$.

\begin{enumerate}[\rm 1.]\everymath{\displaystyle}
\item We say that $M$ is {\em infinitely pointwise $\alpha$-lineable} if, for every $x \in M$, there exists a family $\mathcal{M} = \{W_k\}_{k \in \N}$ of vector subspaces such that for each $k \in \N$:

\begin{enumerate}[\rm (i)]\everymath{\displaystyle}
\item $\dim (W_k) = \alpha$,
\item $x \in W_k \subset M \cup \{0\}$, and
\item $W_k \cap W_l = {\rm span} \{x\}$ for any $l \in \N$ with $l \neq k$.
\end{enumerate}

\item If additionally $X$ is endowed with a topology and each vector space $W_k \in \mathcal{M}$ ($k \in \N$) is dense in $X$, we say that $M$ is {\em infinitely pointwise $\alpha$-dense lineable} in $X$.
\end{enumerate}
\end{definition}

It is not difficult to get infinite vector subspaces from a vector space of infinite dimension. Indeed, we have only to divide it in an adequate way. So $\alpha$-lineability implies ``infinite" $\alpha$-lineability. The same property happens to be true for the pointwise case. Let us include, in order to be self-contained, the proof of this.

\begin{proposition} \label{infinitepointwise}
Let $X$ be a vector space, $\alpha \geq \aleph_0$, and $M \subset  X$.
Then $M$ is pointwise $\alpha$-lineable if, and only if, it is infinitely pointwise $\alpha$-lineable.
\end{proposition}

\begin{proof}The only if part is obvious since, in general, infinitely pointwise notions imply ordinary pointwise notions. So let us proof the if part.

Since $M \subset X$ is pointwise $\alpha$-lineable, for each $x \in M$ there is a vector space $W \subset M \cup \{0\}$ such that $\textrm{dim}(W) = \alpha$ and  $x \in W$. So, there exists a set $I$ with $\textrm{card}(I) = \alpha$ and $\{ \omega_i \, : \, i \in I\}$ such that $x$ and $\omega_i$'s are linearly independent and  $W = \textrm{span}( \{ \omega_i \, : \, i \in I\} \cup \{x\})$. Now, since $\alpha\ge\aleph_0$, we can split $I$ into infinitely many pairwise disjoint subsets $I_k$ ($k\in\N$), each one with cardinality $\alpha$. Thus, by considering the vector spaces $W_k$ given by
$$W_k := \textrm{span} \left( \{\omega_i \, : \, i \in I_k\} \cup \{x\} \right),$$
we have that $W_k \subset M \cup \{0\}$, $\textrm{dim}(W_k) = \alpha$ for every $k \in \N$ and, because of the linear independence of $x$ and $\omega_i's$, $W_k \cap W_l = \textrm{span}\{x\}$ for all $k \neq l$, and so the infinite pointwise $\alpha$-lineability of $M$ is proved.
\end{proof}

\begin{remark} \label{remarkplus}
Observe that from the above proof we have that if a set $M$ is infinitely pointwise $\alpha$-lineable then for any $x \in M$ there exists a family $\mathcal{M} = \{ W_k\}_{k \in \N}$ of vector spaces satisfying conditions (i), (ii) and (iii) of Definition \ref{infinitepointwisedefinition}(a), and additionally the following fourth condition:
$$W_k + W_l \subset M \cup \{ 0 \} \quad \text{ for any } k, l \in \N .$$
\end{remark}

The case of dense-lineability is not so clear, since denseness may not be inherited by the infinitely many vector spaces constructed. Recall that (see \cite[Definition 2.1]{arongarciaperezseoane2009} or \cite[Theorem 2.1]{bernalordonez2014}) if $M$ and $N$ are subsets of some vector space $X$, then $M$ is said to be stronger than $N$ if $M + N \subset M$.

%
\begin{theorem}
\label{infinitelypointwiselineabilitytodense}
Let $X$ be a metrizable separable topological vector space, and $\alpha$ be an infinite cardinal number, and $M$ be a nonempty subset of $X$ for which there is a nonempty subset $N$ of $X$ such that
\begin{enumerate}[\rm (i)]
\item $M$ is stronger than $N$;
\item $M\cap N=\varnothing$;
\item $N$ is dense-lineable.
\end{enumerate}
If $M$ is pointwise $\alpha$-lineable, then $M$ is infinite pointwise $\alpha$-dense lineable (and therefore pointwise $\alpha$-dense lineable).
\end{theorem}

\begin{proof}
 Since $X$ is separable there exists a sequence $(x_n)_n \subset X$ such that the set $\{x_n \, : \, n \in \N\}$ is dense in $X$, where we can assume without loss of generality that $x_1 = 0$.

Now, since $M$ is pointwise $\alpha$-lineable, by Proposition \ref{infinitepointwise} it is infinitely pointwise $\alpha$-lineable. Hence, for every $x \in M$, there exists a family $\mathcal{M} = \{W_k\}_{k \in \N}$ of vector spaces such that $x \in W_k \subset M \cup \{0\}$ with $\textrm{dim}(W_k) = \alpha$ for every $k \in \N$, and $W_k \cap W_l = \textrm{span}\{x\}$ for every $k \neq l$. By Remark \ref{remarkplus}, we assume that $W_k + W_l \subset M \cup \{ 0 \}$ for any $k, l \in \N $.

Since each $W_k$ is a vector space, for each $k \in \N$ there exists a linearly independent set $\{\omega_i^{(k)} \, : \, i \in I\}$ with $\textrm{card}(I) = \alpha$ and $1\in I$ such that
\[ W_k = \textrm{span} \{\omega_i^{(k)} \, : \, i \in I\},\]
where we can assume without loss of generality that $\omega_1^{(k)} = x$ for every $k \in \N$.

Due to the fact that $\alpha$ is an infinite cardinal, we can split $I$ into infinitely many pairwise disjoint nonempty sets $I_n$ ($n \in \N$), where $1\in I_1$.

Now, fix $k,n \in \N$ and $i \in I_n$. Since multiplication by scalars is a continuous operation in a topological vector space, there exists $\varepsilon^{(k)}_i > 0$ such that
\[ d(\varepsilon^{(k)}_i \omega^{(k)}_i, 0) < \dfrac{1}{n}, \]
where $d$ denotes a fixed translation invariant metric on $X$.

On the other hand, $N$ is dense lineable in $X$, so there exists a vector subspace $V \subset N \cup \{0\}$ with $V$ dense in $X$. Now, for each $n \in \N$, the denseness of $V$ guarantees the existence of $v_n \in V$ such that
\[d(v_n,x_n) < \dfrac{1}{n},\]
where we can choose $v_1 := 0$ (recall that $x_1=0$).

Now, define the elements $x^{(k)}_{n,i}$ as follows:
\[x^{(k)}_{n,i} := v_n + \varepsilon^{(k)}_i \omega^{(k)}_i, \quad k, n \in \N, i \in I_n,\]
so that we consider the vector space $W^{(k)}$ generated by them, that is:
\[W^{(k)} := \text{span} \{x^{(k)}_{n,i} \, : \, n \in \N, i \in I_n\}.\]

We will show that for every $k\in\N$, $x \in W^{(k)}$, $W^{(k)}$ is dense in $X$, $W^{(k)} \subset M \cup \{0\}$, $\textrm{dim}(W^{(k)}) = \alpha$ and $W^{(k)}\cap W^{(l)}=\textrm{span}\{x\}$ for any $l\in\N$ with $l\ne k$.

From now on, let $k\in\N$ fixed.

\begin{enumerate}[\rm (1)]\everymath{\displaystyle}
\item Since $1 \in I_1$ we have that:
\[v_1 + \varepsilon^{(k)}_1 \omega^{(k)}_1 = \varepsilon^{(k)}_1 \omega^{(k)}_1 = \varepsilon^{(k)}_1 x \in W^{(k)}.\]
Thus,
\[x = \omega^{(k)}_1 = \dfrac{1}{\varepsilon^{(k)}_1} (v_1 + \varepsilon^{(k)}_1 \omega^{(k)}_1) \in W^{(k)}.\]

\item Now, in order to prove the density of $W^{(k)}$ in $X$, let us fix $n \in \N$ and take some $i_n \in I_n$. By considering $u^{(k)}_n := x^{(k)}_{n,i_n}$, we have that
\begin{eqnarray*}
d(u^{(k)}_n, x_n) & \leq & d(u^{(k)}_n,v_n) + d(v_n,x_n) \\[1em]
& = & d(v_n + \varepsilon^{(k)}_{i_n} \omega^{(k)}_{i_n}, v_n) + d(v_n,x_n) \\[1em]
& = & d(\varepsilon^{(k)}_{i_n} \omega^{(k)}_{i_n},0) + d(v_n,x_n) \\[1em]
& < & \dfrac{1}{n}+\dfrac{1}{n} = \dfrac{2}{n} \to 0 \qquad (n \to \infty).
\end{eqnarray*}
Since $(x_n)_n$ is dense in $X$, we get that $(u^{(k)}_n)_n$ is also dense in $X$, and the same holds true for $W^{(k)}$.

\item Fix $\omega \in W^{(k)} \setminus \{0\}$. There are scalars $c_1, c_2, \ldots, c_s$ with $c_s \neq 0$, as well as indices $i_r \in I_r$ ($r = 1, 2, \ldots, s$), such that
$$\omega = c_1 x^{(k)}_{1,i_1} + c_2 x^{(k)}_{2,i_2} + \ldots + c_s x^{(k)}_{s,i_s}.$$
But by the definition of $(x^{(k)}_{n,i})_{n,i}$ we have that
$$\omega = y_0 + z^{(k)}_0,$$
where $$y_0 := c_1 v_1 + c_2 v_2 + \ldots + c_s v_s,$$
$$z^{(k)}_0 := c_1 \varepsilon^{(k)}_{i_1} \omega^{(k)}_{i_1} + c_2 \varepsilon^{(k)}_{i_2} \omega^{(k)}_{i_2} + \ldots + c_s \varepsilon^{(k)}_{i_s} \omega^{(k)}_{i_s}.$$
Recall that $v_1, v_2, \ldots, v_s \in V$, which is a vector space, so $y_0 \in V \subset N \cup \{0\}$. Analogously, $\omega^{(k)}_{i_1}, \omega^{(k)}_{i_2}, \ldots, \omega^{(k)}_{i_s} \in W_k$, they are linearly independent, and $c_s \varepsilon^{(k)}_{i_s} \neq 0$, so $z_0 \in W_k \setminus \{0\} \subset M$. If $y_0 = 0$, then $\omega = z_0 \in M$. If $y_0 \neq 0$, then $\omega = y_0 + z_0 \in N  + (W_k \setminus \{0\}) \subset N + M \subset M$ because $M$ is stronger than $N$, and we have $W^{(k)} \subset M \cup \{0\}$.

\item Let us show now that $\textrm{dim}(W^{(k)}) = \alpha$. For this, it is clear that
$$\textrm{card}\big( \{ (n,i) \, : \, n \in \N, i \in I_n\} \big) = \textrm{card} \left( \bigcup_{n=1}^\infty I_n \right) = \textrm{card}(I) = \alpha.$$
So, if we prove that the vectors of $\{ x^{(k)}_{n,i} \, : \, n \in \N, i \in I_n\}$ are linearly independent we are done. Indeed, assume by way of contradiction that
$c_1x^{(k)}_{1,i_1}+c_2x^{(k)}_{2,i_2} + \ldots + c_sx^{(k)}_{s,i_s} = 0$ with $c_s \neq 0$. As done before (and following the same notation), we have that $y_0 + z^{(k)}_0 = 0$, where $y_0 \in V $ and $z^{(k)}_0 \in W_k \setminus \{0\}$. But then, $y_0 = - z^{(k)}_0 \in W_k \setminus \{0\}$, since $W_k$ is a vector space. Hence, we have that
$$y_0 \in  \big(W_k \setminus \{0\}\big) \cap V  \subset M \cap N = \varnothing,$$
which is a contradiction.

\item It only remains to prove that $W^{(k)} \cap W^{(l)} = \textrm{span} \{x\}$ for every $l \neq k$.

With this aim, let $\omega \in W^{(k)} \cap W^{(l)}$. Since $\omega$ is in each of the vector spaces, we can write it as:
\[
\begin{array}{cclcl}
\omega & = & \alpha \cdot x_{1,1}^{(k)}+\sum_{s=1}^p \alpha_{s} x_{n_s,i_s}^{(k)} & = & \alpha\cdot \varepsilon_1 + \sum_{s=1}^p \alpha_{s} v_{n_s} + \sum_{s=1}^p \alpha_{s} \varepsilon_{i_s}^{(k)} \omega_{i_s}^{(k)}, \\[1.5em]

\omega & = & \beta \cdot x_{1,1}^{(l)}+ \sum_{r=1}^q \beta_{r} x_{n_r,j_r}^{(l)} & = & \beta \cdot \varepsilon_1 + \sum_{r=1}^q \beta_{r} v_{n_r} + \sum_{r=1}^q \beta_{r} \varepsilon_{j_r}^{(l)} \omega_{j_r}^{(l)},
\end{array}
\]
where $(n_s,i_s) \ne (1,1) \ne (n_r,j_r)$ for any $s$, $r$. Hence
\[
\sum_{s=1}^p \alpha_{s} v_{n_s} - \sum_{r=1}^q \beta_{r} v_{n_r} = (\beta - \alpha ) \varepsilon_1 x + \sum_{r=1}^q \beta_{r} \varepsilon_{j_r}^{(l)} \omega_{j_r}^{(l)} - \sum_{s=1}^p \alpha_{s} \varepsilon_{i_s}^{(k)} \omega_{i_s}^{(k)}.
\]
Observe that the left hand side is in $V \subset N \cup \{0\}$, and the right hand side is in $W_k+ W_l \subset M \cup \{0\}$, and since $M \cap N = \varnothing$, each term of the above equality must be zero. So,
\[
\alpha \varepsilon_1 x +\sum_{s=1}^p \alpha_{s} \varepsilon_{i_s}^{(k)} \omega_{i_s}^{(k)} = \beta  \varepsilon_1 x + \sum_{r=1}^q \beta_{r} \varepsilon_{j_r}^{(l)} \omega_{j_r}^{(l)} =: \gamma \cdot x ,
\]
because the left hand side is in $W_k$, the right hand is in $W_l$ and $W_k \cap W_l = {\rm span}\{x\}.$ Now, $\varepsilon$'s are nonnull and the linear independence of the $x$, $\omega_i^{(k)}$'s and of the $x$, $\omega_j^{(l)}$'s, gives us $\alpha = \beta =\frac{\gamma}{\varepsilon_1}$, and $\alpha_s =0 = \beta_r$. Thus, $c= \gamma \cdot x \in {\rm span}\{x\}$, as required.
\end{enumerate}
\end{proof}

\begin{remark}
  Observe that, under the hypotheses of the above theorem, (pointwise) $\alpha$-\emph{dense} lineability implies \emph{infinite} (pointwise) $\alpha$-\emph{dense} lineability. In fact, although there exist many examples of dense-lineable sets $M$, for many of them there exists a set $N$ enjoying conditions (i), (ii) and (iii) of Theorem \ref{infinitelypointwiselineabilitytodense} (see \cite[\S 7.3]{aronbernalpellegrinoseoane2016}).

  Up to now, we do not know if this fact remains true for a general dense-lineable set $M$ in a general topological vector space $X$. So, we propose the following question.
\end{remark}

\begin{openproblem}
  Let $X$ be a topological vector space and $M\subset X$ be a (pointwise) $\alpha$-dense lineable set. Is $M$ always infinite (pointwise) $\alpha$-dense lineable?
\end{openproblem}

\section{Linear and topological structures of the set of continuous, unbounded and integrable functions on $\R^N$}

Let $N \in \N$. Throughout this section, we use the following notation:
\begin{enumerate}
\item $C^\infty (\R^N)$ represents the set of all real functions on $\R^N$ that are infinitely many times differentiable on $\R^N$. This becames a Fr\'echet space when endowed with the topology of uniform convergence on compacta for all partial derivatives of all orders, see \cite{Horvath}.
    \item $L^p(\R^N)$  ($p \in [1, + \infty)$) denotes the vector space of all (classes of) functions $\R^N \to \R$ that are $p$-integrable Lebesgue on $\R^N$. This becomes a Banach space under the $p$-norm
        $$||f||_{L^p} := \left( \int_{\R^N} |f|^p \, dx_1 \cdots dx_N \right)^{1/p}.$$
        \item For each multi-index $\alpha = ( \alpha_1, \dots , \alpha_N)\in (\N \cup \{0\})^N$, we set $|\alpha | := \alpha_1 + \cdots + \alpha_N$.
\item For each $x \in \R^N$, $||x||$ will stand for the classical euclidean norm on $\R^N$.
\end{enumerate}

From now on we consider the space of functions $X$ given by
        $$X:= C^\infty (\R^N) \cap \bigcap_{p \geq 1} L^p (\R^N).$$
    Observe that the formula
    $$||f||_m = \max_{|\alpha | \leq m} \sup_{ ||x|| \leq m} |D^\alpha f(x) | \qquad (f \in C^\infty (\R^N), \ m=1,2,\dots)$$
    defines an increasing sequence of seminorms generating the natural Fr\'echet topology of $C^\infty (\R^N)$. Here $D^\alpha$ denotes the partial differential operator of order $\alpha$. With this and the fact that $L^p (\R^N) \cap L^q (\R^N) \subset L^r (\R^N)$ whenever $1 \leq p \leq r \leq q < + \infty$, we can consider a natural translation invariant metric $d_X$ in $X$ given by:
$$d_X(f,g):=  \sum_{m=1}^\infty \frac{1}{2^m} \cdot \frac{||f-g||_{m}}{1 + ||f-g||_{m}} + \sum_{p=1}^\infty \frac{1}{2^p} \cdot \frac{||f-g||_{L^p}}{1 + ||f-g||_{L^p}}.$$
We have that $(X,d_X)$ is a Fr\'echet space and convergence in $d_X$ is equivalent to uniform convergence on compacta for all partial derivatives of all orders and convergence in $p$-norm for every $p \in [1, + \infty)$.

In this space of functions $X$ we shall search for unbounded functions in a pre-fixed not relatively compact subset $A \subset \R^N$. Let us show first that we can always find such a function.

\begin{example} \label{example}

Let $A \subset \R^N$ be a not relatively compact set in $\R^N$. Then there exists a sequence $(a_n)_n \subset A$ such that $\|a_n\|$ strictly increases to $+\infty \quad (n \to \infty)$, and $\| a_{n+1} - a_n\| >1$. Therefore, the closed balls $\overline{B}(a_n, \frac{1}{2^{n}})$ are pairwise disjoint for every $n \in \N$. Now, by the Smooth Urysohn's Lemma \cite[Corollary 1.7.1]{shastri2011}, there exist bump functions $\Phi_n : \R^N \to \R$ such that for each $n \in \N$:
\begin{enumerate}[\rm (a)]
\item $\Phi_n \in C^\infty(\R^N)$,
\item $\Phi_n(x) \in [0,1]$ for all $x \in \R^N$,
\item $\Phi_n(x) = 0$ for all $x \notin \overline{B}(a_n,\frac{1}{2^{n}})$, and
\item $\Phi_n(x) = 1$ for all $x \in \overline{B}(a_n,\frac{1}{2^{n+1}})$.
\end{enumerate}
In particular, since each $\Phi_n$ is bounded and has compact support, we have that $\Phi_n \in L^p(\R^N)$ for every $p \geq 1$, and so $\Phi_n \in X$ for each $n \in \N$.
Now, consider the function
\[ w(x) := \sum_{n=1}^\infty n \Phi_n(x), \quad x \in \R^N. \]
Since the supports of the $\Phi_n$'s are pairwise disjoint, the expression of $w$ does not actually represent an infinite series, but rather each one of the bump functions, that is:
\[ w(x) = \left\{
\begin{array}{cl}
n \Phi_n(x) & \text{if $x \in B(a_n,\frac{1}{2^{n}})$, $n \in \N$}, \\[1em]
0                  & \text{otherwise.}
\end{array}
\right. \]
Thus, $w \in C^\infty(\R^N)$. Moreover,
\begin{eqnarray*}
\| w \|^p_{L^p} & = & \sum_{n=1}^\infty n \| \Phi_n \|^p_{L^p} \leq \sum_{n=1}^\infty n \frac{1}{2^{nN}}< +\infty,
\end{eqnarray*}
so $w \in L^p(\R^N)$ for all $p \geq 1$. Thus $w \in X$. Finally, $w$ is unbounded on $A$. Indeed,
\[ |w(a_n)| = \left| \sum_{m=1}^\infty m \Phi_m(a_n) \right| = n  \Phi_n(a_n)= n \to +\infty \qquad (n \to \infty). \]
\end{example}

From now on, given a not relatively compact subset $A$ in $\R^N$, we denote:
$$nBC^\infty I(A):= \{ f \in X: f \textrm{ is unbounded in } A\}.$$
The main result of this section shows that this set is not only nonempty but even maximal pointwise spaceable. Recall that dim$(X)=\mathfrak{c}$. We will explicitly construct the closed vector space.

\begin{theorem}\label{pointwisecspaceability} Let $A$ be not relatively compact in $\R^N$. Then the set $nBC^\infty I(A)$ is pointwise $\mathfrak{c}$-spaceable in $(X,d_X)$.
\end{theorem}

\begin{proof}
Let $f \in  nBC^\infty I(A)$ be fixed. There exists a sequence $(a_n)_n \subset A$  such that $f(a_n) \to \infty$ as $n \to \infty$. Without loss of generality we can assume that $(\| a_n\|)_n$ is strictly increasing to infinity, $\|a_{n+1} - a_n\| \geq 1$ and $|f(a_n)| > 1$ for all $n \in \N$.

For each $n \in \N$, by considering the closed balls,
\[\begin{array}{rcl}
B_{1,n} & := & \overline{B}\left( a_n, \dfrac{1}{|f(a_n)|^{1/N} 2^{n+1}} \right), \\[1em]
B_{2,n} & := & \overline{B}\left( a_n, \dfrac{1}{|f(a_n)|^{1/N} 2^{n+2}} \right),
\end{array}\]
the Smooth Urysohn's Lemma provides a bump function $\Phi_n \in C^\infty(\R^N)$ such that:

\begin{enumerate}[\rm (a)]\everymath{\displaystyle}
\item $\Phi_n(x) \in [0,1]$ for all $x \in \R^N$,
\item $\Phi_n(x) = 0$ for all $x \notin B_{1,n}$, and
\item $\Phi_n(x) = 1$ for all $x \in B_{2,n}$.
\end{enumerate}
In particular,
\[\|\Phi_n\|_{L^p}^p \leq 2^N \left( \dfrac{1}{|f(a_n)|^{1/N}2^{n+1}} \right)^N = \dfrac{1}{|f(a_n)|\cdot 2^{nN}} < 1,\]
and $\Phi_n \in L^p(\R^N)$ for every $p \geq 1$. Then the $\Phi_n$'s are in $X$ and the supports are pairwise disjoint.

Now, we consider a partition of $\N$ into infinitely many pairwise disjoint subsequences:
\[ \N = \{ j(n) \, : \, n \in \N\} \cup \bigcup_{k = 1}^\infty \{i(n,k)\,:\,n\in\N\},\]
where the sequences $(j(n))_n$ and $(i(n,k))_{n,k}$ are strictly increasing in $n$ and $k$.

For each $k \in \N$, we define the functions $f_k : \R^N \to \R$ by
\[f_k(x) := \sum_{n=1}^\infty f(a_{i(n,k)}) \Phi_{i(n,k)}(x).\]
Observe that, as the supports of $\Phi_{i(n,k)}$'s are pairwise disjoint, for each $x \in \R^N$, there exists a neighbourhood $U$ of $x$, and $n_0, k_0 \in \N$ such that:
\[f_k(x) = f(a_{i(n_0,k_0)}) \Phi_{i(n_0,k_0)}(x), \quad x \in U.\]
So, $f_k \in C^\infty(\R^N)$ for all $k \in \N$, and if we compute its $L^p$-norm, we obtain that
\begin{eqnarray*}
\|f_k\|_{L^p}^p & = & \sum_{n=1}^\infty |f(a_{i(n,k)})|  \|\Phi_{i(n,k)}\|^p_{L^p} \\[1em]
& \leq & \sum_{n=1}^\infty |f(a_{i(n,k)})| \left( \dfrac{1}{|f(a_{i(n,k)})|^{1/N} 2^{i(n,k)}} \right)^N \\[1em]
& = & \sum_{n=1}^\infty \dfrac{1}{2^{i(n,k)N}} \leq  \sum_{n=1}^\infty \dfrac{1}{2^{nN}} \leq 1.
\end{eqnarray*}
Hence, $f_k \in L^p(\R^N)$ for each $p \geq 1$, and $f_k \in X$ ($k \in \N$). Furthermore, for each $k \in \N$,
\[ |f_k(a_{i(n,k)})| = |f(a_{i(n,k)})| \to +\infty \quad (n \to \infty).\]
Thus, the sequence of functions $(f_k)_k \subset nBC^\infty I(A)$.

Let $\ell_1$ be the Banach space of all $1$-summable real sequences. Now, we can define the operator $T : \ell_1 \to nBC^\infty I(A) \cup \{ 0 \}$ given by
\[T((\alpha_k)_{k=0}^\infty) := \alpha_0 f + \sum_{k=1}^\infty \alpha_k f_k.\]
Indeed, since the supports of the $f_k$'s are pairwise disjoint, given $x_0 \in \R^N$, there exists $k_0 \in \N$, and a neighbourhood of $x_0$ where
\[T((\alpha_k)_k) (x) = \alpha_0 f(x) + \alpha_{k_0} f_{k_0}(x), \]
so $T((\alpha_k)_k) \in C^\infty(\R^N)$. On the other hand,
\begin{eqnarray*}
\|T((\alpha_k)_k) \|_{L^p} & \leq & |\alpha_0|\cdot  \|f\|_{L^p} + \sum_{k=1}^\infty |\alpha_k|\cdot  \| f_k\|_{L^p} \\[1em]
& \leq & |\alpha_0| \cdot  ||f||_{L^p} + \sum_{k=1}^\infty |\alpha_k| < \infty.
\end{eqnarray*}
It is clear that $T((0,0,\dots))=0$, If $(\alpha_k)_k$ is not the null sequence, then $T((\alpha_k)_k)$ is not bounded in $A$ because either $\alpha_0 \neq 0$, and then
\[|T((\alpha_k))(a_{j(n)})| = |\alpha_0| |f(a_{j(n)})| + 0 \to +\infty \quad (n \to \infty);\]
or $\alpha_0 = 0$, and then there exists $k_0 \in \N$ such that $\alpha_{k_0} \ne 0$ and
\[ |T((\alpha_k)_k)(a_{i(n,k_0)})| = |\alpha_{k_0}||f(a_{i(n,k_0)})| \to +\infty \quad (n \to \infty).\]
Thus the operator $T$ is well defined and injective. Then, for the vector subspace
$$W_f := T(\ell_1)$$
we have
\[ f = T((1,0,0,\ldots))  \qquad {\rm and}  \qquad \textrm{dim}(W_f) = \mathfrak{c}.\]
Observe that this shows that $nBC^\infty I(A)$ is pointwise $\mathfrak{c}$-lineable.

Now, to get the pointwise $\mathfrak{c}$-spaceability it is enough to show that the closure $\overline{W_f}$ of $W_f$ in $(X,d_X)$ satisfies that
\[\overline{W_f} \backslash \{0\} = \overline{T(\ell_1)} \backslash \{0 \} \subset nBC^\infty I(A).\]
For this, consider $h \in \overline{T(\ell_1)} \backslash \{0\}$. Then, there exists $(H_l)_{l } \subset W_f \backslash \{0\}$ such that $H_l \to h$ as $l \to \infty$ in $(X,d_X)$. So, for each $l \in \N$ we can write:
\[H_l = \alpha_0^l f + \sum_{k=1}^\infty \alpha_k^l f_k.\]
Since $(H_l)_l$ converges to $h$ on $(X, d_X)$, we have convergence on compacta in $\R^N$. Therefore, for each $n \in \N$, by considering the singleton (actually compact set) $\{a_{j(n)}\}$ we have that:
\[H_l(a_{j(n)}) = \alpha_0^l f(a_{j(n)}) \to h(a_{j(n)}) \quad (l \to \infty),\]
and
\[\alpha_0^l \to \dfrac{h(a_{j(n)})}{f(a_{j(n)})} =: \alpha_0 \quad (l \to \infty).\]
Hence,
\[h(a_{j(n)}) = \lim_{l \to \infty} H_l(a_{j(n)}) = \alpha_0 f(a_{j(n)}).\]
Now, we arrive at two possible situations depending on $\alpha_0$:
\begin{enumerate}[\rm (i)]\everymath{\displaystyle}

\item If $\alpha_0 \neq 0$ we are done, since $|h(a_{j(n)})| \to \infty$ as $n \to \infty$, and then $h \in nBC^\infty I(A)$.

\item If $\alpha_0 = 0$, let us fix $k_0 \in \N$. For each $n \in \N$, consider the compact set $K_n$ given by
\[ K_n := B_{1,i(n,k_0)} =\overline{B}\left( a_{i(n,k_0)}, \dfrac{1}{|f(a_{i(n,k_0)})|^{1/N} 2^{i(n,k_0)+1}} \right).\]
For every $x \in K_n$ we have:
\[ H_l(x) = \alpha_0^l f(x) + \sum_{k=1}^\infty \alpha_k^l f_k(x) = \alpha_0^l f(x) + \alpha_{k_0}^l f_{k_0}(x) \to h(x) \quad (l \to \infty).\]
Hence, by taking $x = a_{i(n,k_0)}$, and recalling that $\alpha_0 = 0$, we have that
\[ \alpha_{k_0}^l f_{k_0}(a_{i_{(n,k_0)}}) =  \alpha_{k_0}^l f_{k_0}(a_{i(n,k_0)}) + \alpha_0^l f(a_{i(n,k_0)}) \to  h(a_{i(n,k_0)}) \qquad (l \to \infty).\]
Thus, we have
\[  \alpha_{k_0}^l \to \dfrac{h(a_{i(n,k_0)})}{f(a_{i(n,k_0)})} =: \alpha_{k_0} \quad (l \to \infty).\]
Observe that we can assume that we fixed a $k_0$ for which $\alpha_{k_0} \ne 0$. Otherwise, we have that $\alpha_{k_0} =0 $ for any $k_0 \in \N$ and, as $\alpha_0=0$, $H_l \to 0 $ ($l \to \infty$) pointwise in $\R^N$. This is a contradiction because $H_l \to h \ne 0$ ($l \to \infty$) in $(X,d_X)$.

We can then evaluate $h$ at these points, obtaining:
\[|h(a_{i(n,k_0)})| = |\alpha_{k_0}|\cdot | f(a_{i(n,k_0)})| \to +\infty \quad (n \to \infty).\]
Thus $h \in nBC^\infty I(A)$, and so $\overline{W_f} \backslash \{0\} \subset nBC^\infty I(A)$ as desired.
\end{enumerate}
\end{proof}

As a direct consequence of this result, we have shown the following:

\begin{corollary}\label{pointwiseclineability}
The family $nBC^\infty I(A)$ is maximal pointwise lineable in $X$.
\end{corollary}

Now, by Proposition \ref{infinitepointwise}:

\begin{corollary}\label{infinitelypointwiseclineability}
The family $nBC^\infty I(A)$ is infinitely pointwise $\mathfrak{c}$-lineable in $X$.
\end{corollary}

If we take the topological structure of the space $X$ into account, the characterization provided in Theorem  \ref{infinitelypointwiselineabilitytodense} let us get density also.
\begin{theorem}\label{pointwisedenseclineability}
The family $nBC^\infty I(A)$ is infinitely pointwise $\mathfrak{c}$-dense lineable in $(X,d_X)$.
\end{theorem}

\begin{proof}
Recall that the set $N := C_c^\infty (\R^N)$ of smooth functions with compact support in $\R^N$ is a dense vector space in $(X, d_X)$. If we take $M:=nBC^\infty I(A)$ it is clear that $M+N \subset M$ and $M\cap N = \varnothing$ because both are subsets of $X$, the functions in $M$ are unbounded, and the ones in $N$ are bounded. By Corollary \ref{infinitelypointwiseclineability}, $M$ is infinite pointwise $\mathfrak{c}$-lineable in $X$. Thus, an application of Theorem \ref{infinitelypointwiselineabilitytodense} completes the proof.
\end{proof}

\section{Final remarks}
\begin{enumerate}
\item[{\rm (1)}] From the proof of Theorem \ref{pointwisecspaceability} we can deduce the following result:

\vspace{0.2cm}
{\it For any pre-fixed not relatively compact subset $A$ of $\R^N$, the set
$$nBC^\infty (A) :=\{ f \in C^\infty (\R^N): \ f \text{ is unbounded in } A\}$$
is maximal pointwise spaceable in $C^\infty (\R^N)$.}

\vspace{0.2cm}
In fact, for every $f \in nBC^\infty (A)$, there exists a $\mathfrak{c}$-dimensional subspace $V_f$ such that
$$\overline{{\rm span}(f) \oplus V_f} \subset nBC^\infty (A) \qquad {\rm and} \qquad V_f \subset \bigcap_{p \geq 1} L^p(\R^N).$$

\item[{\rm (2)}] From Corollary \ref{infinitelypointwiseclineability}, it is trivial that the set $nBC^\infty I(A)$ is infinitely pointwise $\mathfrak{c}$-lineable in $C^\infty (\R^N)$. As the set $C_c^\infty(\R^N)$ is dense in $C^\infty(\R^N)$ (endowed with the topology of uniform convergence on compacta for all derivatives of all orders), we have:

\vspace{0.2cm}
{\it The set $nBC^\infty I(A)$ is infinitely pointwise $\mathfrak{c}$-dense lineable in $C^\infty(\R^N)$.}

\vspace{0.2cm}
\item[{\rm (3)}] Let $\alpha: [0, + \infty) \to [1, + \infty)$ be a continuous increasing function. We say that a function $f\in C(\R^N)$ has growth $\alpha$ through the set $A$ whenever
$$\limsup_{\substack{||x|| \to \infty\\x\in A }} \frac{|f(x)|}{\alpha (||x||)} = + \infty .$$
If in Example \ref{example} we modify the definition of the function $w(x)$ as follows:
$$w(x):= \sum_{n=1}^\infty n \cdot \alpha (||a_n||) \cdot \Phi_n(x) \qquad (x \in \R^N),$$
(where in the pre-fixed sequence $(a_n)_n \subset A$ we also assume that $||a_n|| >1$ for any $n\in \N$) we obtain a function in the vector space $X$ that has growth $\alpha$ through the set $A$.

Now, if we consider any fixed function $f $ as above, and we choose a sequence $(a_n)_n \subset A$ such that $\frac{|f(a_n)|}{\alpha (||a_n||)} \to + \infty $ as $n \to \infty$, we can follow all the same steps as in the proof of Theorem \ref{pointwisecspaceability} to obtain the next result:

\begin{theorem} Let $A$ be not relatively compact in $\R^N$ and let $\alpha: [0, + \infty) \to [1, + \infty)$ be a continuous increasing function. Then the set
$$ \{ f \in X: f \text{ has growth } \alpha \text{ through the set } A \}$$
is pointwise $\mathfrak{c}$-spaceable and infinitely pointwise $\mathfrak{c}$-dense lineable in $(X, d_X)$.
\end{theorem}

\item[{\rm (4)}] Let $\Omega$ be an open subset of $\R^N$. We consider
$$ X_\Omega := C^\infty (\Omega) \ \cap \  \bigcap_{p \geq 1} L^p(\Omega),$$
which is a Fr\'echet topological vector space under the metric $d_{X, \Omega}$ defined as
$$d_{X, \Omega}(f,g):=  \sum_{m=1}^\infty \frac{1}{2^m} \cdot \frac{||f-g||_{K_m}}{1 + ||f-g||_{K_m}} + \sum_{p=1}^\infty \frac{1}{2^p} \cdot \frac{||f-g||_{L^p(\Omega)}}{1 + ||f-g||_{L^p(\Omega)}}.$$
where $(K_m)_m$ is an exhaustive sequence of compact subsets in $\Omega$ ($K_m \subset K_{m+1}^\circ$ and $\Omega = \cup_{m=1}^\infty K_m$) and
$$||f||_{K_m} = \max_{|\alpha | \leq m} \sup_{ x \in K_m} |D^\alpha f(x) | \qquad (f \in C^\infty (\Omega), \ m=1,2,\dots).$$
Then the results of Section 3 hold for the set $$nBC^\infty I(A, \Omega):= \{f \in X_\Omega: f \text{ is unbounded in } A\},$$
where $A$ is a not relatively compact subset in $\Omega$. Indeed, for fixed $f\in nBC^\infty I(A, \Omega)$ there exist two sequences $(a_n )_n \subset A$ and $(r_n)_n \subset (0, 1)$ such that:
\begin{enumerate}
\item $(r_n)_n$ is strictly decreasing to zero,
\item $B(a_n, r_n/2) $ is contained in $\Omega$, for each $n \in \N$,
\item $(a_n)_n$ tends to the boundary of $\Omega$ ($n \to \infty$),
\item $||a_{n+1}-a_n|| > r_n$ for each $n \in \N$, and
\item $|f(a_n)| >1$ for each $n \in \N$.
\end{enumerate}
For any $n \in \N $ we consider the closed balls:
$$B_{1,n} ;= \overline{B} \left( a_n, \frac{r_n}{|f(a_n)|^{1/N} \cdot 2^{n+1}} \right),$$
$$B_{2,n} ;= \overline{B} \left( a_n, \frac{r_n}{|f(a_n)|^{1/N} \cdot 2^{n+2}} \right).$$
All these balls are contained in $\Omega$ and are pairwise disjoint. Now we can follow the same steps as in the proof of Theorem 3.2 to get maximal pointwise spaceability of $nBC^\infty I(A, \Omega)$.

\item[{\rm (5)}] Finally, we can apply our Theorem \ref{infinitelypointwiselineabilitytodense} to establish the infinite pointwise dense-lineability of any set for which its pointwise lineability is already known and for which we can find a suitable set $N$. For instance:
    \begin{itemize}
      \item In \cite{pellegrinoraposo2021} it is proved the pointwise $\mathfrak{c}$-lineability of $\ell_p(X)\setminus\bigcup_{q<p}\ell_q(X)$ (where $X$ is any Banach space); taking $N:=c_{00}(X)$, we get its infinite pointwise $\mathfrak{c}$-dense lineability in $\ell_p(X)$.
      \item In \cite{favaropellegrinoraposoribeiroPREPRINT} it is proved that the set $\mathcal{A}_0$ of sequences of continuous unbounded and integrable functions in $[0,+\infty)$ that goes to zero both in $L_1$-norm and uniformly in compacta of $[0,+\infty)$ is pointwise $\mathfrak{c}$-lineable. Taking as $N$ the set $c_{00}(B)$ of Lemma 3.1 in \cite{calderongerlachprado2019}, we also get the infinite pointwise $\mathfrak{c}$-dense lineability of $\mathcal{A}_0$ in $c_0(L_1[0,+\infty)\cap\mathcal{C}[0,+\infty))$.
    \end{itemize}

\end{enumerate}

\bigskip

\bigskip

{\scriptsize
\begin{flushleft}
M.C.~Calder\'on-Moreno and J.A.~Prado-Bassas\\
Departamento de An\'alisis Matem\'atico and Instituto de Matem\'aticas IMUS.  \\
Facultad de Matem\'aticas, Universidad de Sevilla. \\
Avda.~Reina Mercedes s/n, 41012 Sevilla, Spain.  \\
E-mail: {\tt mccm@us.es} and {\tt bassas@us.es}

P.J.~Gerlach-Mena \\
Departamento de Estad\'{\i}stica e Investigaci\'on Operativa.  \\
Facultad de Matem\'aticas, Universidad de Sevilla. \\
Avda.~Reina Mercedes s/n, 41012 Sevilla, Spain.  \\
E-mail: {\tt gerlach@us.es}
\end{flushleft}}
\end{document}